\documentclass[12pt]{amsart}
\usepackage{amssymb}
\usepackage{verbatim}
\usepackage{amscd}

\theoremstyle{plain}
\newtheorem{theorem}{Theorem}
\newtheorem{corollary}{Corollary}
\newtheorem{proposition}{Proposition}
\newtheorem{lemma}{Lemma}

\theoremstyle{definition}
\newtheorem{definition}{Definition}

\theoremstyle{remark}
\newtheorem{remark}{Remark}
\newtheorem{example}{Example}
\newtheorem*{notation}{Notation}

\newcommand{\wh}{\widehat}

\newcommand{\ol}[1]{\overline{#1}}

\def\bbC{\mathbb C}

\def\bbN{\mathbb N}

\def\bbT{\mathbb T}
\def\bbZ{\mathbb Z}

\newcommand{\fA}{\mathfrak A}     \newcommand{\sA}{\mathcal A}
     \newcommand{\sB}{\mathcal B}
     
     \newcommand{\sD}{\mathcal D}

     \newcommand{\sG}{\mathcal G}

\newcommand{\fM}{\mathfrak M}     \newcommand{\sM}{\mathcal M}
\newcommand{\fN}{\mathfrak N}     \newcommand{\sN}{\mathcal N}
     \newcommand{\sO}{\mathcal O}

     \newcommand{\sS}{\mathcal S}

\newcommand{\al}{\alpha}
\newcommand{\be}{\beta}
\newcommand{\vpi}{\varphi}
\newcommand{\si}{\sigma}
\newcommand{\Si}{\Sigma}

\newcommand{\om}{\omega}

\begin{document}

\title[Dirichlet Operator Algebras]
{Representations of Dirichlet Operator Algebras}

\author{Justin R. Peters}
\address{Department of Mathematics\\
    Iowa State University, Ames, Iowa, USA}

\email{peters@iastate.edu}

 \keywords{Dirichlet operator algebra, C$^*$-envelope, Cartan pair, essentailly principal etale groupoid.}
\subjclass[2010]{Primary 47L30, 22A22; Secondary 46L55, 46L30}
 \begin{abstract}
A Dirichlet operator algebra is a nonself-adjoint operator algebra $\sA$ with the property that $\sA + \sA^*$ is norm-dense in the C$^*$-envelope of $\sA.$
We show that, under certain restrictions, $\sA$ has a family of completely contractive representations $\{\pi_i\}$ with the property that the invariant subspaces of $\pi_i(\sA)$ are totally
ordered, and such that, for all $a \in \sA, \ ||a|| = \sup_i ||\pi_i(a)||.$ The class of Dirichlet algebras includes strongly maximal triangular AF algebras, certain
semicrossed product algebras, and gauge-invariant subalgebras of Cuntz C$^*$-algebras.

The main tool is the duality theory for essentially principal etale groupoids.
 \end{abstract}
\maketitle

\section{Introduction}
One of the earliest results in the theory of C$^*$-algebras was that, given a C$^*$-algebra $\fA,$ there are sufficiently many irreducible representations to separate the points.
Equivalently, for $ a \in \fA, \ ||a|| = \sup_{\pi \text{ irreducible}} ||\pi(a)||.$
However, for the larger class of operator algebras, which need not be closed under the adjoint operation, a coordinate-free description
became available only in 1990 (\cite{BRS90}), and hence the question how such abstract algebras of operators can be representated in Hilbert space has become meaningful.
Now the analogue of the C$^*$-algebra result, that there are sufficiently many irreducible representations, does not hold in the larger context of operator algebras, 
as one can see even from finite-dimensional upper triangular matrix algebras.

A number of authors have considered nest representations;  a nest represenations of an operator algebra $\sA$ is a bounded representation $\pi: \sA \to \sB(H)$ 
 for which the lattice of invariant subspaces is linearly ordered.  One can ask, given an operator algebra $\sA,$ if there are sufficiently many nest representations to 
separate the elements of $\sA,$ or to achieve the norm. It seems that the tools do no exist at present to approach the problem in this generality.

Much of the previous work on nest representations of operator algebras has dealt with the class of TAF operator algebras
(i.e., triangular subalgebras of AF C$^*$-algebras); In \cite{DonHopEtAl00}, Theorem 2.4, it is shown that
the kernel of a nest representation is a meet-irreducible ideal.  The meet-irreducible ideals are ideals $I$ with the property that if $I = J \cap K,$ then $I = J$ or $I = K.$
In \cite{DavKatPet03} the converse is proved, namely that the kernel of a bounded nest representation is a meet-irreducible ideal. Some results in the literature have 
topologized the set of meet-irreducible ideals, making an analogy with the primite ideals in a C$^*$-algebra.  In \cite{Lam97} it is shown that for certain TAF algebras, 
every closed ideal is the intersection of meet irreducible ideals. Applying this result
to the $(0)$ ideal, we can conclude that for these TAF algebras there are sufficiently many nest representations to separate points of the algebra.
However, even for this class of TAF algebras, the relation between the norm in the algebra, and the norms of the (bounded) nest represenations is not clear.

These and related ideas have been investigated by various authors.
(cf \cite{DonPitPow2001}, \cite{Lam96}, \cite{Lam97}, \cite{DavKat02}, \cite{HopPetPow00})  

Thus, progress on the question of nest representations hinged on the study of the meet-irreducible ideals. In the current work,
 rather than examing the \emph{internal} structure of operator algebras,
our approach is to look at what might be called the \emph{external} structure: this consists of a C$^*$-cover $\fA$ of the operator algebra $\sA,$ along with
an abelian subalgebra $\fM \subset \sA$ such that the pair $(\fA, \fM)$ forms a Cartan pair in the sense of \cite{JRen08}. Indeed, at the time of these earlier investigations,
the Renault-Kumjian duality theory of essentially principal etale groupoids had not yet been fully developed.  

Here we consider Dirichlet algebras. These are (nonself-adjoint) unital operator algebras $\sA$ with the property that, if we view $\sA$ (and also $\sA^*$) as embedded in a
C$^*$-cover $\fA,$ then $\sA + \sA^*$ is norm-dense in $\fA.$ We also assume that the masa $\fM$ is contained in the Dirichlet algebra $\sA.$

This context of Dirichlet algebras which contain a masa which is a Cartan masa in a C$^*$-cover encompasses, for example, the class of strongly maximal TAF-algebras.
However, it also includes certain semicrossed product algebras,  some subalgebras of Cuntz C$^*$-algebras, and more generally various subalgebras of graph C$^*$-algebras.

The term ``Dirichlet algebra'' is used differently in various contexts. Classicly, a commutative Banach algebra $\sA$ is Dirichlet if it is a subalgebra of $C(X)$
for some compact Hausdorff space $X$ and $\sA + \sA^*$ is dense in $C(X),$ or equivalently, if the real parts of the functions in $\sA$ are dense in the real parts of the
functions in $C(X).$  In our setting, the Dirichlet algebras will be noncommutative.  Thus, $C(X)$ is replaced by a noncommutative C$^*$-algebra, a C$^*$-cover.
While we require that $\sA \cap \sA^*$ contain an abelian subalgebra which is a masa $\fM$ in a C$^*$-cover of $\sA,$ we
do not require the stricter condition that $\sA \cap \sA^* = \fM.$

By \cite{JRen08} there is a unique faithful conditional expectation $E: \fA \to \fM.$ Now some authors (e.g. Arveson, \cite{Arv.Not11}) require
the conditional expectation to be multiplicative when restricted to $\sA;$ this is not required in our context, though it is satisfied in some of the motivating examples.

In \cite{DavKat11} the authors defined a unital operator algebra $\sA$ to be \emph{semi-Dirichlet} if $\sA^*\sA \subset \ol{\sA + \sA^*}$ (where $\sA,\ \sA^*$ are viewed
as subsets of C$^*_{env}(\sA)$).  A natural question is whether the results here can be extended to the semi-Dirichlet case, though clearly the tools employed here
would not be applicable, as there need not be a Cartan masa.

In section~\ref{s:BandN} we introduce some of the groupoid notation and briefly refer to the main results of Kumjian and Renault. In section~\ref{s:reps of Cartan pairs}, we first prove
that if $(G, \Si) $ is a twisted etale Hausdorff locally compact second countable essentially principal groupoid, and $x_0 \in X, $ the unit space, that the state of $C(X),\ f \mapsto f(x_0)$
extends uniquely to a state of the groupoid C$^*$-algebra if and only if $x_0$ has trivial isotropy. (Corollary~\ref{c:unique extension}) The corresponding GNS irreducible
representation $\pi_0$ maps the Cartan pair $(\fA, \fM)$ (where $\fA$ is the reduced C$^*$-algebra of the twised etale groupoid, and $\fM$ is the Cartan masa) 
to a Cartan pair $(\pi_0(\fA), \pi_0(\fM)).$ (Corollary~\ref{c:Cartan}) Furthermore, $\pi_0(\fM)$ is weakly dense in a masa in $\sB(H_{\pi_0}).$ (Corollary~\ref{c:fM is a masa})
Proposition~\ref{p:dirichlet} implies that the restriction of $\pi_0$ to the Dirichlet algebra $\sA \subset \fA$ has the property that its invariant subspaces are totally ordered. The
main result, that there are sufficiently many such representations to achive the norm, is Corollary~\ref{c:sufficiently many nest}

 In section~\ref{s:examples} we describe some examples of Dirichlet algebras: strongly maximal TAF algebras, certain semicrossed product algebras, Dirichlet subalgebras of
Cuntz algebras $\sO_n,$ and more generally ``gauge invariant'' subalgebras of certain graph C$^*$-algebras.

\vspace{.25cm}
\section{Background and Notation} \label{s:BandN}

\begin{definition} \label{d:Cartan}
A pair $(\fA, \fM)$ consisting of a separable C$^*$-algebra $\fA$ and an abelian subalgebra $\fM$ is called a \emph{Cartan pair} if
\begin{enumerate}
\item $\fM$ contains an approximate unit of $\fA;$
\item $\fM$ is maximal abelian C$^*$-subalebra of $\fA;$
\item $\fM$ is regular;
\item there exists a faithful conditional expectation $E: \fA \to \fM.$ 
\end{enumerate}
\end{definition}

In the setting here, the C$^*$-algebras $\fA$ will be unital, so the first condition is vacuous. 
To explain the regularity condition we need to introduce the concept of normalizers. 
\begin{definition} \label{d:normalizer}
An element $0 \neq n \in \fA$ is said to \emph{normalize} $\fM$ if 
\[ n^* \fM n \subset \fM \text{ and } n \fM n^* \subset \fM .\]
\end{definition}
We will denote the set of normalizers by $\sN.$  Trivially, every $h \in \fM$ is a normalizer. There are masas in C$^*$-algebras whose only normalizers 
belong to the masa.  Such masas are called singular.  We are interested in the setting in which the set of normalizers is large.

\begin{definition} \label{d:regular}
We say that the masa $\fM \subset \fA$ is \emph{regular} if the set $\sN$ of normalizers generate $\fA$ as a C$^*$-algebra.
\end{definition}

We will make use of the results of \cite{JRen08}. In particular, if $(\fA, \fM)$ is a Cartan pair, there is an etale, Hausdorff, locally compact, second countable and
essentially principal groupoid $G,$ and a twist $\Si$ over $G,$ such that the elements of $\fA$ can be viewed as sections of the line bundle on $(G, \Si),$ 
or equivalently as functions $f: \Si \to \bbC,$ satisfying $f(z\si) = \bar{z}f(\si)$ for $\si \in \Si,\ z \in \bbT.$
Furthermore, the product of elements in $\fA$ can be viewed as convolution on the twised groupoid.

The elements of the groupoid $G$ can be expressed as triples $[x, g, y]$ where $g$ is a partial homeomorphism of the space $X = \wh{\fM},$ 
the Gelfand space of the masa $\fM \subset \fA,$ where $g$ belongs to the Weyl pseudo-group $\sG$. The domain of $g$ is open in $X, \ y \in \text{dom}(g),$
and $ x = g(y).$ If $g, h$ belong to the Weyl pseudo-group $\sG$ and if $y \in \text{dom}(g) \cap \text{dom}(h)$ with $g(y) = h(y),$ then
$[x, g, y] = [x, h, y].$ Thus the triple $[x, g, y]$ depends on the germ of the partial homeomorphism.

Consider the set 
\[\{(x, n, y) \in X \times \sN \times X : n^*n(y) > 0 \text{ and } x = \al_n(y) \}.\]
Define an equivalence relation on the set by $(x, n, y) \sim (x', n', y')$ if and only if $y' = y$ and there are functions $h, h' \in C(X) := \fM$ such that
$h(y),\ h'(y) > 0$ and $nh = n'h'.$

The elements of the twist $(G, \Si)$ are equivalence classes $[x, n, y].$   For each $n \in \sN$ there is a partial homeomorphism $\al_n$ of $X$
satisfying $n^*hn = h\circ \al_n \, n^*n, \ h \in C(X).$ The Weyl pseudo-group is precisely $\{ \al_n: n \in \sN\},$ where $\sN$ is the set of normalizers.

Now $\bbT = \{z \in \bbC: |z| = 1\}$ acts on $(G, \Si)$ by $z\cdot[x, n, y] = [x, zn, y].$ The range and source maps are local homeomorphism from $(G, \Si) \to G^{(0)},$
where $G^{(0)}$ is the unit space, given by $r[x, n, y] = x,\ s[x, n, y] = y.$

We will make repeated use of Theorem~5.6 of \cite{JRen08} without further comment; this is the "duality theorem" that every Cartan pair can be realized 
as a twisted groupoid C$^*$-algebra. There is an isomorphism  $ a \in \fA \mapsto \hat{a}$ where
\[ \hat{a}([y, n, x]) = \frac{E(n^*a)(x)}{\sqrt{n^*n(x)}} \]

\begin{notation} 
For simplicity of notation will will write $\hat{a}(x, n, y)$ in place of $\hat{a}([x, n, y]).$
\end{notation}

The elements of the C$^*$-algebra $\fA$ can now be viewed as a subalgebra of the
 continuous functions on $(G, \Si)$ vanishing at infinity, satisfying $ \hat{a}(x, zn, y) = \bar{z}\, \hat{a}(x, n, y),$ or more formally, sections of the line bundle over $G$.

The convolution is given by 
\[f*g(\si) = \sum_{s(\tau) = s(\si)} f(\si \tau^{-1}) g(\tau) \]
If we write $\si = [x, n, y],$ and $s^{-1}(y) = \{ [w_i, m_i, y]\},$ this can be expressed as

\[ \wh{f*g}(x, n, y) = \sum_i f(x, nm_i^*, w_i)g(w_i, m_i, y) \]
The involution is given by $f^*(x, n, y) = \ol{f(y, n^*, x)} .$

\begin{remark} \label{r:mult by C(X)}
If follows from the definition of convolution that if $f \in C(X),$ which is to say that $f$ is supported on the unit space $G^{(0)},$ and $a \in \fA$ is arbitrary that
$\wh{fa}(x, n, y) = f(x) \hat{a}(x, n, y)$ and $\wh{af}(x, n, y) = \hat{a}(x, n, y) f(y).$
\end{remark}

\begin{notation} \label{n:compact support}
Following \cite{JRen08}~Lemma 5.5, we will denote by $\sN_c $ the set of normalizers with compact support, and by
  $\fA_c$ its linear span.  By the Lemma, $\sN_c$ is dense in $\sN,$ and $\fA_c$ is dense in $\fA.$
\end{notation}
	
If $(\fA, \fM)$ is a Cartan pair, let $X $ be the Gelfand space of $\fM.$  A point $x \in X$ is said to have \emph{trivial isotropy} if 
 $[x, n, x] \in (G, \Si)$ implies $[x, n, x] \in G^{(0)},$ the unit space. If, on the other hand, there exists $n \in \sN$ with $E(n)(x) = 0$ and $[x, n, x] \in (G, \Si),$ then $x$ is said to have
non-trivial isotropy. (cf \cite{JRen08}~Lemma 5.2)

\begin{definition} \label{d:pseudoorbit}
Extending the terminology of \cite{JRen08}, if $\sG$ is the Weyl pseudogroup associated with the Cartan pair $(\fA, \fM),$ and $x_0 \in X := \hat{\fM}, $ we say that
\[ \sG(x_0) := \{g(x_0): g \in \sG, \ x_0 \in \text{dom}(g)\} \]
is the \emph{pseduoorbit} of $x_0.$
\end{definition}

\begin{lemma} \label{l:Ginvariant}
\begin{itemize}
\item[(i)] If $x_0 \in X,$ the pseudoorbit $\sG(x_0)$ is $\sG$-invariant.
\item[(ii)] The closure $\ol{\sG(x_0)}$ of the pseudoorbit $\sG(x_0)$ is $\sG$-invariant.
\item[(iii)] The complement $X\setminus \ol{\sG(x_0)}$ is $\sG$-invariant.
\end{itemize}
\end{lemma}

\begin{proof}
We only prove (ii). Let $y $ be in the closure of the pseudo-orbit $\sG(x_0),$ so there is a sequence $\{y_k\} \subset \sG(x_0)$ converging to $y.$
Now if $x = g(y),$ we claim $x \in \ol{\sG(x_0)}.$ Since $\text{dom}(g)$ is an open neighborhood of $y,$ it follows that there is an $K \in \bbN$ so that
$y_k \in \text{dom}(g)$ for $k \geq K.$ By (i), $g(y_n) \in \sG(x_0),$ and $x = g(y) = \lim_k g(y_k) \in \ol{\sG(x_0)}.$
\end{proof}

\begin{lemma} \label{l:n*n and nn*}
For any $n \in \sN,$ \ and $x \in \text{dom}(\al_n), \ nn^*(\al_n(x)) = n^*n(x).$
\end{lemma}

\begin{proof}
\begin{align*}
nn^*(\al_n(x)) \cdot n^*n(x) &= n^*(nn^*)n(x) \\
	&= (n^*n)^2(x)
\end{align*}
By \cite{JRen08}, Proposition 4.6,  $\text{dom}(\al_n) = \{x: n^*n(x) > 0\}.$ Since $x$ is in the domain of $\al_n$, we obtain the result.
\end{proof}

\begin{corollary} \label{c:Ginvariant}
Let $n \in \sN$ and $x_0 \in X.$ Then $n^*n$ vanishes on the  pseudoorbit $\sG(x_0)$ if and only if $nn^*$ vanishes on the pseudoorbit.
\end{corollary}

\begin{proof}
Suppose $nn^*$ vanishes on $\sG(x_0),$ and $y = \al_m(x_0) $ for some $m \in \sN.$  By Lemma~\ref{l:n*n and nn*}, $n^*n(y) 
= (nn^* \circ \al_n) \circ \al_m(x_0) = nn^*\circ \al_{nm}(x_0) =0.$ So, $n^*n$ vanishes on $\sG(x_0).$

For the other direction, interchange the roles of $n^*$ and $n.$
\end{proof}

\begin{lemma} \label{l:hat n}
Let $n,\ m \in \sN$ and $y \in \text{dom}(\al_m). $ Then
\[ \hat{n}(x, m, y) \neq 0 \iff \al_m = \al_n \text{ on some neighborhood of } y. \]
\end{lemma}

\begin{proof}
By definition, $\hat{n}(x, m, y) \neq 0$ if and only if $E(m^*n)(y) \neq 0.$  Since $m^*n$ is a normalizer, and hence its open support is a bisection (\cite{JRen08}, Proposition 4.7),
$E(m^*n) \neq 0$ is equivalent to asserting that $m^*n$ lies in the unit space.  But then,
$ \text{id}_{\sO} = \al_{m^*n} = \al_m^* \circ \al_n = (\al_m)^{-1} \circ \al_n,$ where $\sO$ is the domain of $m^*n,$  equivalently, $\al_n = \al_m$
on the intersection of their domains.

Conversely, if $\al_n, \ \al_m$ coincide on some nonempty open set $\sO,$ it follows that $n, m$ are nonzero on the bisection $S = \{ [\al_n(y), \al_n, y]
= [\al_m(y) \al_m, y], y \in \sO \}.$ (\cite{JRen08}, Corollary 5.3)  Hence $\hat{n}(x, m, y) \neq 0.$

\end{proof}

\begin{corollary} \label{c:fn}
For $n \in \sN, \ f \in C(X), \ fn = n f\circ \al_n .$
\end{corollary}

\begin{proof}
By Remark~\ref{r:mult by C(X)},  for $f,\ g \in C(X),$ we have
$\wh{fn}(x, m, y) = f(x)\hat{n}(x, m, y) = f\circ \al_m(y) \hat{n}(x, m, y)$ and $\wh{ng}(x, m, y) = \hat{n}(x, m, y)g(y) = g(y)\hat{n}(x, m, y).$
By Lemma~\ref{l:hat n}, $\hat{n}(x, m, y) \neq 0 \iff  \al_m = \al_n$ in some neighborhood of $y.$ Thus, taking $g = f\circ \al_n,$ we have
$\wh{fn} = \wh{nf\circ \al_n},$ or $fn = nf\circ \al_n.$
\end{proof}

Observe that Proposition 4.6 of \cite{JRen08} follows from Corollary~\ref{c:fn} by left multiplying by $n^*.$

\begin{lemma} \label{l:E(n^*an)}
If $a \in \fA,$ and $n \in \sN,$ then $E(n^*an) = n^*E(a)n.$
\end{lemma}

\begin{proof} If we can prove the result holds on a dense subset of $\fA,$ then the result will follow from the continuity of $E.$
Thus, we take $a \in \fA_c,$ so $a$ has the form $a = \sum_{i=0}^N  n_i$ where $N \in \bbN,$ and the $n_i \in \sN$ and $n_0 = E(a) .$
Then $n^*an = \sum_{i=0}^N n^*n_in$ and we claim that $E(n^*an) = n_0\circ \al_n n^*n.$  Let $ 0 < i \leq N.$
Now the partial homeomorphism $\al_{n_i}$ is not the identity on any open subset of $X,$ otherwise $E(n_i)$ would be nonzero.
If $n^* n_i n \neq 0,$ then by \cite{JRen08}, Lemma 4.9, $\al_{n^* n_i n} = \al_{n}^{-1} \circ \al_{n_i}\circ \al_n$ is not the identity on any open
subset of its domain, so that $E(n^*n_i n) = 0.$ This verifies the claim, and the lemma follows.
\end{proof}

\begin{lemma} \label{l:isotropy preserved}
Let $x_0 \in X$ and $y$ be in the pseudo-orbit of $x_0,$ so $y = \al_n(x_0)$ for some $n \in \sN.$
Then $y$ has trivial isotropy if and only if $x_0$ has trivial isotropy.
\end{lemma}

\begin{proof}
Suppose $x_0$ has non-trivial isotropy, so there exists $m \in \sN$ with $ E(m) = 0$ and $\al_m(x_0) = x_0$ Then
\[ \al_{nmn^*}(y) =  \al_n \circ \al_m \circ \al_{n^*}(y) = \al_n \circ \al_m(x_0) = \al_n(x_0) = y \]
and $E(nmn^*) = nE(m)n^* = 0,$ so that $y$ has non-trivial isotropy.  The converse is proved by reversing the roles of $y$ and $x_0.$
\end{proof}

\subsection{Dirichlet algebras} \label{ss:Dirichlet}

For a nonsef-adjoint operator algebra $\sA,$ we say that a C$^*$-algebra $\fA$ is a \emph{C$^*$-cover} if there is a completely isometric embedding $j: \sA \mapsto \fA.$
By \cite{Ham79}, \cite{VPaul86} for given an operator algebra $\sA$  there is a unique minimal C$^*$-cover, called the C$^*$-envelope, with the propery that
if $\fA$ is any C$^*$-cover there is a surjective C$^*-$homomorphism $\tau: \fA \to $ C$^*_{env}(\sA)$ such that  the restriction of $\tau$  to $j(\sA) $
is completely isometric. If $\fA$ is a C$^*$-cover of $\fA,$ to simplify notation we will generally write $\sA \subset \fA.$ 

\begin{definition} \label{d:Dirichlet}
Recall from the Introduction that an operator algebra $\sA$ is \emph{Dirichlet} if $\sA + \sA^*$ is dense in the C$^*$-envelope of $\sA.$
If $\fA$ is a C$^*$-cover for $\sA,$ and if $\sA + \sA^*$ is dense in $\fA,$ we will say that $\sA$ is \emph{Dirichlet in $\fA.$}
\end{definition}
It follows immediately that if $\sA$ is Dirichlet in some C$^*$-cover $\fA,$ then it is Dirichlet.
In the examples of section~\ref{s:examples}, the operator algebra $\sA$ arises naturally as a subalgebra of C$^*$-algebra $\fA$
for which $\sA + \sA^*$ is dense in $\fA.$ It suffices to work with the C$^*$-cover.

The following result is similar to Proposition 0.1 of \cite{OrrPet95}, though the context is different.
\begin{proposition} \label{p:dirichlet}
Let $\sA$ be a unital operator algebra with C$^*$-cover $\fA$ such that $\sA$ is Dirichlet in $\fA.$ Let $\pi: \fA \to \sB(H)$ be
an irreducible representation of $\fA.$ If the weak closure of $\pi(\sA)$ contains a masa, then the lattice $\text{Lat}(\pi(\sA))$ is a nest, and
$\pi(\sA)$ is weakly dense in a nest algebra.
\end{proposition}

\begin{proof} 
Let $\sM \subset \sB(H)$ be a masa contained in $\ol{\pi(\sA)}^{\text{wk}}$ and $P_1,\ P_2$ be two projections in $\sB(H)$ invariant under $\pi(\sA),$
and let $F$ be a projection in $\sM.$
If $\{a_{\nu}\}$ is a net in $\pi(\sA)$ converging weakly to $F,$ then since $\pi(a_{\nu})P_i = P_i \pi(a_{\nu})P_i $ for all $\nu,$
we have that $FP_i = P_iFP_i, \ i = 1, 2.$ Now $P_i^{\perp}$ is invariant for $\pi(\sA^*), $ so that $\pi(a_{\nu}^*)P_i^{\perp} = P_i^{\perp}\pi(a_{\nu}^*)P_i^{\perp},$
so that we obtain $FP_i^{\perp} = P_i^{\perp}FP_i^{\perp},$ or equivalently, $P_iF = P_iFP_i, \ i = 1, 2.$  Thus we conclude $P_iF = FP_i,$ and since $\sM$
is generated by its projections, $P_i$ commutes with $\sM$ and so $P_i \in \sM, \ i = 1, 2.$  It follows that $P_1,\ P_2$ commute.

We have shown that the lattice of invariant projections is commutative, and now we claim it is in fact a nest; that is, there is a total ordering on the lattice of projections.
To this end, suppose that is not the case, so there are projections $P_1, P_2 \in \text{Lat}(\pi(\sA))$ so that $0 \neq F_1:= P_1^{\perp}P_2$ and $0 \neq F_2 := P_1P_2^{\perp}.$
Let $h_1, $ (resp., $h_2$) be a nonzero vector in the range of $F_1$ (resp., $F_2$). Since $\pi$ is an irreducible representation of $\fA,$ every nonzero vector is cyclic. Thus 
there is $b \in \fA$ so that $<h_2, \pi(b)h_1> \neq 0.$ Since $\sA + \sA^*$ is dense in $\fA,$ there are $a_1, a_2 \in \sA$ such that $<h_2, \pi(a_1 + a_2^*)h_1> \neq 0.$
But then
\begin{align*}
<h_2, \pi(a_1 + a_2^*)h_1> &= <h_2, \pi(a_1)h_1> + <h_2, \pi(a_2)^*h_1> \\
	&= <h_2, \pi(a_1)h_1> + <\pi(a_2)h_2, h_1> \\
	& = 0
\end{align*}
since $\pi(a_1)h_1$ belongs to the range of $P_2,$ while $h_2$ is orthogonal to the range of $P_2,$ and $\pi(a_2)h_2 $ belongs to the range of $P_1,$ while $h_1$ is
orthogonal to the range of $P_1.$

This contradiction establishes the claim.  Since Lat$(\pi(\sA))$ is a nest,  $\pi(\sA)$ is weakly dense in a nest algebra follows from Corollary 12.12 of \cite{Dav88}.
\end{proof}

\begin{notation} \label{n:representation}
In section~\ref{s:MASAS and reps} we will deal with representations of non-selfadjoint operator algebras $\sA.$  While in the literature a representation is always
a homomorphism $\pi: \sA \to \sB(H)$ for some Hilbert space $H,$ there are various conditions imposed on $\pi:$ e.g., boundedness, contractive, completely contractive.
Here a representation will always be a  \emph{completely contractive} homomorphism $\pi: \sA \to \sB(H).$  Indeed, these representations arise as restrictions to the
operator alagebra $\sA$ of representations of a C$^*$-cover.
\end{notation}


\section{Representations of Cartan pairs} \label{s:reps of Cartan pairs}

Next we show that if $x_0 \in X$ has non-trivial isotropy, then the state of $C(X),\ f \mapsto f(x_0)$ does not have a unique extension to a state of $\fA.$

Now if $x_0$ has non-trivial isotropy, then there is an $n \in \sN$ with $E(n)(x_0) = 0$ and $[x_0, n, x_0] \in (G, \Si).$ Define the linear functional $\vpi$ on $\fA$ by
\[ \vpi(a) = \hat{a}(x_0, n, x_0), \ a \in \fA.\]
\begin{lemma} \label{l:isotropy}
Let $\rho_0$ denote that state of $\fA$ given by $\rho_0(a) = E(a)(x_0).$  Let $a \in \fA^+.$ Then
\[ |\vpi(a)| \leq \rho_0(a) .\]
\end{lemma}

\begin{proof}
\begin{align*}
|\vpi(a)| &= |\hat{a}(x_0, n, x_0)| \\
	&= \frac{|\rho_0(n^*a)|}{\sqrt{n^*n(x_0)}} \\
	&= \frac{|\rho_0(an)|}{\sqrt{n^*n(x_0)}} \\
	&= \frac{|\rho_0(a^{\frac{1}{2}}\, (a^{\frac{1}{2}}n))|}{\sqrt{n^*n(x_0)}} \\
	&\leq \frac{\sqrt{\rho_0(a)} \, \sqrt{\rho_0((a^{\frac{1}{2}}n)^* (a^{\frac{1}{2}}n))}}{\sqrt{n^*n(x_0)}} \\
	&\leq  \frac{\sqrt{\rho_0(a)} \, \sqrt{\rho_0(n^*an)}}{\sqrt{n^*n(x_0)}} \\
	&\leq \frac{\sqrt{\rho_0(a)} \, \sqrt{E(n^*an)(x_0)}}{\sqrt{n^*n(x_0)}} \\
	&\leq \frac{\sqrt{\rho_0(a)} \, \sqrt{(n^*E(a)n)(x_0)}}{\sqrt{n^*n(x_0)}} \\
	&\leq  \frac{\sqrt{\rho_0(a)} \, \sqrt{E(a)\circ \al_n(x_0) \, n^*n(x_0)}} {\sqrt{n^*n(x_0)}}\\
	&\leq \sqrt{\rho_0(a)} \, \sqrt{E(a) \circ \al_n(x_0)}
\end{align*}

But since $\al_n$ fixes $x_0,$ we have that $E(a) \circ \al_n(x_0) = E(a)(x_0) = \rho_0(a).$
Thus we obtain $ |\vpi(a)| \leq \sqrt{\rho_0(a)} \, \sqrt{\rho_0(a)} = \rho_0(a).$
\end{proof}

If $\vpi$ is a bounded linear functional on $\fA,$ then $\vpi^*$ is the linear functional given by $\vpi^*(a) = \ol{\vpi(a^*)}$ (\cite{KRI}, p. 255).
It clear that $|\vpi^*(a)| = |\vpi(a)|, \ a $ Hermitian. Furthermore, $\vpi(a) + \vpi^*(a)$ is real-valued if $a \in \fA$ is Hermitian. From Lemma~\ref{l:isotropy} we have
that for $a \in \fA^+, \  |\vpi(a) + \vpi^*(a)| \leq |\vpi(a)| + |\vpi^*(a)| \leq 2\rho_0(a).$ Hence, for $a \in \fA^+, \ 2\rho_0(a) + \vpi(a) + \vpi^*(a) \geq 0.$  Thus
\begin{proposition} \label{p:states and isotropy}
With notation as above, the linear functional $\vpi_0 = \rho_0 + \frac{1}{2}(\vpi + \vpi^*) $ is a state of $\fA$ which extends the state of $C(X), \ f \in C(X) \mapsto f(x_0).$
Thus if $x_0$ has non-trivial isotropy, the pure state of $C(X),\ f \mapsto f(x_0),$ has at least two distinct extensions to states of $\fA.$
\end{proposition}

\begin{proof}
We have shown that $\vpi_0$ is a positive linear functional.  Since $\vpi(1) = \vpi^*(1) = 0,$ it follows that $\vpi_0$ is a state of $\fA.$ Finally, note that both
$\rho_0$ and $\vpi_0$ extend the state of $C(X), \ f \in C(X) \mapsto f(x_0)$ since $\vpi(f) = \vpi^*(f) = 0$ for $f \in C(X).$

Since the linear functional $\vpi$ is not the zero functional, the states $\rho_0, \ \vpi_0$ are distinct.
\end{proof}

Next we prove the converse of Proposition~\ref{p:states and isotropy}.

\begin{proposition} \label{p:converse}
If the state of $C(X), \ f \mapsto f(x_0),$ does not have a unique extension to a state of $\fA,$ then $x_0$ has non-trivial isotropy.
\end{proposition}

\begin{proof}
Let $\rho_0$ denote the state of $\fA$ given by $\rho_0(a) = E(a)(x_0),$ and let $\vpi$ denote a state of $\fA$ which extends the state of $C(X),\ f \mapsto f(x_0),$ 
and which is distinct from $\rho_0.$

Claim: There exists $n \in \sN,$ with $E(n)(x_0) = 0,$ such that $\vpi(n) \neq 0.$ Suppose not; then for any element $a \in \fA_c, \ a = \sum_{i=0}^N n_i$ with
$n_i \in \sN, N \in \bbN, $ and $E(a) = n_0,$ we have that
\[ \vpi(a) = \sum_{i=0}^N \vpi(n_i) =\vpi(n_0) = n_0(x_0) = E(a)(x_0) = \rho_0(a) \]
so that $\vpi$ and $\rho_0$ coincide on $\fA_c,$ and hence on all of $\fA.$  Since this contradicts the assumption that $\vpi$ and $\rho_0$ are distinct, it must 
then be the case that $\vpi(n) \neq 0$ for some $n \in \sN.$

By the Cauchy-Schwarz inequality, we have that
\[  0 < |\vpi(n)| = |\vpi(1\cdot n)| \leq \sqrt{\vpi(n^*n)} = \sqrt{n^*n(x_0)} \]

Now let $f \in C(X),\ 0 \leq f \leq 1 = f(x_0),$ and observe that
\begin{align*}
 \vpi(n) &= \vpi([f + (1-f)]n[f + (1-f)]) \\
	&= \vpi(fnf) + \vpi(fn(1-f)) + \vpi((1-f)nf) + \vpi((1-f)n(1-f)) \\
	&= \vpi(fnf)
\end{align*}

To see this, we again appeal to Cauchy-Schwarz:
\begin{align*}
 |\vpi(fn(1-f))|  &= \vpi((n^*f)^*(1-f)) \\
	&\leq \sqrt{\vpi((fn)^*(fn))} \sqrt{\vpi((1-f)^2)} \\
	&= \sqrt{\vpi((fn)^*(fn))} \sqrt{(1-f)^2(x_0)} \\
	&= 0 
\end{align*}
Note that the third line makes use of the assumption that $\vpi,$ restricted to $C(X),$ is evaluation at $x_0.$
The other terms, $\vpi((1-f)nf)$ and $\vpi((1-f)n(1-f))$ are treated similarly.

We need to show that $x_0$ has non-trivial isotropy. Suppose not; then $\al_n(x_0) \neq x_0.$
Now since $f \in C(X)$ is only constrained by $ 0 \leq f \leq 1 = f(x_0),$ we can choose $f$ to be supported in a neighborhood of $x_0$ sufficiently small
so that $f(x) f\circ \al_n(x) = 0$ for all $x \in X.$ But then, appealing to Corollary~\ref{c:fn},
\[ 0 \neq \vpi(n) = \vpi(fnf) = \vpi(n f\circ \al_n f) = 0 \]
which is absurd. It follows that $\al_n(x_0) = x_0,$ so that $x_0$ is a point of non-trivial isotropy.
\end{proof}

\begin{corollary} \label{c:unique extension}  Let $x_0 \in X.$  Then the state $f \in C(X) \mapsto f(x_0)$ extends to a unique state of $\fA$ if and only if $x_0$
has trivial isotropy.
\end{corollary}

\begin{corollary} \label{c:irreducible rep}
If $(\fA, \fM)$ is a Cartan pair, and $x_0$ a point of $X = \hat{\fM}$ with trivial isotropy, then
\begin{itemize}
\item[(i)] The pure state $f \in \fM \mapsto f(x_0)$ of $\fM$ has a unique extension $\rho_0$  to a pure state of $\fA.$
\item[(ii)] The GNS representation $\pi_0$ constructed from $\rho_0$ is irreducible.
\end{itemize}
\end{corollary}

\begin{proof}
The first statement follows from Corollary~\ref{c:unique extension}, since if a pure state has a unique extension to a state of $\fA,$ that extension is  pure.
The second statement follows from the fact that the GNS representation constructed from a pure state is irreducible.
\end{proof}

In \cite{Kum86} Kumjian defined the unique extension property: if $(\fA, \fM = C(X))$ is a Cartan pair, then the unique extension property holds if each pure state of $C(X)$
extends uniquely to a pure state of $\fA.$ Kumjian showed that if the corresponding groupoid $G$ is principle, then the unique extension property holds. In \cite{JRen08},
Proposition 5.8, Renault proved the converse. These results follow from Corollary~\ref{c:unique extension}. Indeed, every point $x \in X$ extends uniquely to a state of $\fA$
if and only if there is no point of non-trivial isotropy, which is to say that the groupoid $G$ is principal. In particular, Corollary~\ref{c:unique extension} implies Proposition 5.8 of \cite{JRen08}.

\begin{lemma} \label{l:ker pi0}
Let $x_0, \ \rho_0, \ \pi_0$ be as in Corollary~\ref{c:irreducible rep}
For $a \in \fA,\ a \in \text{ker}(\pi_0)$ if and only if $\hat{a}(z, n, y) = 0$ for all $y \in \sG(x_0)$ and $n \in \sN$ such that $y \in \text{dom}(\al_n).$
\end{lemma}

\begin{proof}
$\pi_0(a) = 0$ if and only if $\pi_0(a)\xi = 0$ for all $\xi \in H_0,$ the Hilbert space associated to $\pi_0$ by the GNS construction. If $b \to [b]$ is the 
canonical map of $\fA$ into $H_0,$
since the image is dense, it is enough that $\pi_0(a)[b] = 0 $ for all $b \in \fA.$  But since $\sN$ generates a dense linear subspace of $\fA,$ it is enough to know that
$\pi_0(a)[n] = 0$ for all $n \in \sN.$  But, $\pi_0(a)[n] = 0$ if and only if $\rho_0(n^*a^*an) = 0.$ By Lemma~\ref{l:E(n^*an)} and the definition of $\rho_0,$ this is
$n^*E(a^*a)n(x_0) = E(a^*a)(\al_n(x_0)) n^*n(x_0),$ which vanishes if and only if $E(a^*a)$ vanishes on the pseudo-orbit of $x_0.$ Finally, that is the case if and only if
\[\wh{a^*a}(y, 1, y) = 0  \text{ for all } y \in \sG(x_0) .\]

Now,
\begin{align*}
\wh{a^*a}(y, 1, y) &= \sum_{[z, n, y] \in s^{-1}(y)} \wh{a^*}(y, n^*, z)\, \hat{a}(z, n, y) \\
	&= \sum_{[z, n, y] \in s^{-1}(y)} |\hat{a}(z, n, y)|^2
\end{align*}
since $\wh{a^*}(y, n^*, z) = \ol{\hat{a}(z, n, y)}.$
Thus, $a \in \text{ker}(\pi_0)$ if and only if $\hat{a}(z, n, y) = 0$ for all $y \in \sG(x_0), \ n \in \sN$ such that $y \in \text{dom}(\al_n).$
\end{proof}

\begin{corollary} \label{c:ker pi0 cap C(X)}
For $f \in C(X), \ f \in \text{ker}(\pi_0)$ if and only if $f$ vanishes on the pseudo-orbit $\sG(x_0).$
Furthermore, if $a \in \text{ker}(\pi_0),$ then $E(a) \in \text{ker}(\pi_0).$
\end{corollary}

For the remainder of this section, we will assume that $x_0, \ \rho_0, \ \pi_0$ are as in Corollary~\ref{c:irreducible rep}.

\begin{proposition} \label{p:pi0M is a masa}
$\pi_0(\fM)$ is a masa in $\pi_0(\fA).$
\end{proposition}

\begin{proof}
Let $a \in \fA$ be such that $\pi_0(a)$ commutes with $\pi_0(\fM).$ Then for any $f \in \fM = C(X), \ fa - af \in \text{ker}(\pi_0).$
By Lemma~\ref{l:ker pi0}, \ $fa - af \in \text{ker}(\pi_0)$ if and only if, for all $y \in \sG(x_0)$ and for all $n \in \sN$ with
$y \in \text{dom}(\al_n), \ (\wh{fa} - \wh{af})(x, n, y) = 0 \  (x = \al_n(y)).$
Let $n \in \sN.$ Computing
\begin{align*}
0 &= (\wh{fa})(x, n, y) - (\wh{af})(x, n, y) \\
	&= f(x)\hat{a}(x, n, y) - \hat{a}(x, n, y) f(y) \\
	&= \hat{a}(\al_n(y), n, y)[f\circ \al_n(y) - f(y)]
\end{align*}
Suppose $n \in \sN$ and $ E(n) = 0.$ 
Since $x_0$ has trivial isotropy, the same holds for $y \in \sG(x_0).$ (Lemma~\ref{l:isotropy preserved})  Hence $\al_n(y) \neq y.$ So for some $f \in C(X), \ f(y) \neq f(\al_n(y)).$  Thus, 
$\hat{a}(\al_n(y), n, y) = 0$ for all $y \in \sG(x_0), \ n \in \sN$ such that $E(n) = 0.$  

Furthermore, if $n = h \in C(X), \ \hat{a}(y, h, y) = \wh{E(a)}(y, h, y),$ since the conditional expectation $E(a) $ is defined by restriction of $a$ to the unit space.
Thus, $(\hat{a} - \wh{E(a)})(x, n, y) = 0$ for all $n \in \sN, \ y \in \sG(x_0) \cap \text{dom}(\al_n).$
It follows from Lemma~\ref{l:ker pi0} that $a - E(a) \in \text{ker}(\pi_0),$ or equivalently, $\pi_0(a) = \pi_0(E(a)).$  In other words, $\pi_0(\fM)$ is a masa in $\pi_0(\fA).$

\end{proof}

We aim to show not only is $\pi_0(\fM)$ a masa in $\pi_0(\fA),$ but that it is a Cartan subalgebra. Recall that a masa is \emph{regular} if the set of its
normalizers generate the C$^*$-algebra.

\begin{lemma} \label{l:pi0M is regular}
$\pi_0(\fM)$ is regular in $\pi_0(\fA).$
\end{lemma}

\begin{proof}
Since the set $\sN$ of normalizers of $\fM$ generate $\fA,$ it follows that $\pi_0(\sN) := \{ \pi_0(n): n \in \sN \}$ generates $\pi_0(\fA).$ Thus it is enough to show that
if $n \in \sN, \  \pi_0(n)$ normalizes the masa $\pi_0(\fM).$ 

\[ \pi_0(n)^* \pi_0(\fM) \pi_0(n) = \pi_0(n^*\fM n) \subset \pi_0(\fM) \]
and similarly $\pi_0(n) \pi_0(\fM) \pi_0(n)^* \subset \pi_0(\fM).$
\end{proof}

Next we define a conditional expectation on $\pi_0(\fA).$ Define $E_0 : \pi_0(\fA) \to \pi_0(\fM)$ by $E_0(\pi_0(a)) = \pi_0(E(a)).$
We need to show this is well-defined. But if $\pi_0(a) = 0,$ then by Lemma~\ref{l:ker pi0}, \ $ \hat{a}(z, n, y) = 0 $ for all $y \in \text{dom}(\al_n)$ such
that $y \in \sG(x_0).$ In particular, taking $n = 1,$ we have that $\hat{a}(y, 1, y) = 0$ for all $y \in \sG(x_0).$  But by definition, $E(a)$ is the restriction of $a$
to the unit space, so that $E(a)(y) = 0$ for all $y \in \sG(x_0).$ But then by Corollary~\ref{c:ker pi0 cap C(X)}, \  $E(a) \in \text{ker}(\pi_0).$  Thus $E_0$ is well-defined
on $\pi(\fA).$

 That $E_0(ca) = cE_0(a) $ and $E_0(ac) = E_0(a)c $ for $a \in \pi_0(\fA), \ c \in \pi_0(\fM)$ are immediate. Thus $E_0$ is a conditional expectation on $\pi_0(\fA).$
It remains to show that the conditional expectation is faithful.

\begin{lemma} \label{l:E0 faithful}
The conditional expectation $E_0: \pi_0(\fA) \to \pi_0(\fM)$ defined above is faithful.
\end{lemma}

\begin{proof}
Suppose, for $a \in \fA, \ E_0(\pi_0(a^*a)) = 0.$ By definition, $E_0(\pi_0(a^*a)) = \pi_0(E(a^*a)).$ Now $E(a^*a)$ is the restriction of $\wh{a^*a}$ to the unit space, and
$\wh{E(a^*a})$ is in the kernel of $\pi_0$ if and only if $E(\wh{a^*a})$ vanishes on the pseudo-orbit $\sG(x_0).$ 
 Thus, $E(a^*a) $ is in the kernel of $\pi_0$ if and only if $\wh{a^*a}(y, 1, y) = 0$ for all $y \in \sG(x_0).$ By the same calculation as in Lemma~\ref{l:ker pi0},
\[ \wh{a^*a}(y, 1, y) = \sum_{[x, n, y] \in s^{-1}(y)} |\hat{a}(x, n, y)|^2 \]
Thus, $\hat{a}(x, n, y) = 0$ for all $[x, n, y] \in s^{-1}(y)$ and for all $y \in \sG(x_0).$  Again by Lemma~\ref{l:ker pi0}, $a \in \text{ker}(\pi_0).$
It follows that $E_0$ is faithful.
\end{proof}

Recall the definition of a Cartan pair (\ref{d:Cartan}).  We have shown
\begin{corollary} \label{c:Cartan}
The pair $(\pi_0(\fA), \pi_0(\fM))$ is a Cartan pair.
\end{corollary}

\begin{remark} \label{r:Cartan}
We have shown that the property of being a Cartan pair is preserved under certain irreducible representations. However, we have no reason to believe
it is preserved under arbitrary irreducible representations. Indeed, if $x_0$ has non-trivial isotropy, then the argument of Proposition~\ref{p:pi0M is a masa} fails, and it is not clear
that $\pi_0(\fM)$ is a masa in $\pi_0(\fA).$
\end{remark}


\section{MASA's in $\sB(H)$ and representations of Dirichlet algebras} \label{s:MASAS and reps}

We now assume that $x_0$ has trivial isotropy, so that \mbox{$\rho_0(a) := E(a)(x_0)$} \ \ $(a \in \fA)$ is a pure state and hence that the corresponding GNS
representation $\pi_0$ is irreducible. Let $\sN(x_0) := \{ n \in \sN: x_0 \in \text{dom}(\al_n) \} = \{ n \in \sN: n^*n(x_0) > 0 \}.$

 We want to obtain a model for the Hilbert space $H_0$ on which $\pi_0$ acts. To this end,
let $n \in \sN(x_0)$ and let $[n]$ be the image of $n$ in $H_0,$ under the canonical map $n \mapsto [n] = n + \fN, $ where $\fN = \{a \in \fA: \rho_0(a^*a) = 0\}.$
We denote the inner product (respectively, the norm) in $H_0$ by $<\cdot, \cdot>$ (respectively, $||\cdot||_2).$

\begin{lemma} \label{l:H0}
\begin{enumerate}
\item If $n \in \sN(x_0),$ then $||n||_2 = \sqrt{n^*n(x_0)}.$
\item If $n \in \sN(x_0)$ with $n^*n(x_0) = 1,$ then $n$ is a  unit vector in $H_0.$ 
\item If $m, \ n \in \sN(x_0)$ is such that $E(m^*n)(x_0) = 0,$ then $[n] \perp [m].$
\item If $n \in \sN(x_0)$ and $f \in \fM = C(X),$ then $[fn] = f\circ \al_n(x_0)[n]$ and $[nf] = f(x_0)[n].$
\end{enumerate}
\end{lemma}

\begin{proof}
For [1], note that $<[n], [n]> = E(n^*n)(x_0) = n^*n(x_0),$ and [2] follows immediately.  \\
For [3], $<[n], [m]> = E(m^*n)(x_0) = 0,$ \\
and for [4], a calculation shows that $||[fn]||_2 = |f\circ \al_n(x_0)| \ ||[n]||_2.$ Now
\begin{align*}
<[fn], [n]> &= E(n^*fn)(x_0)  \\
	&= f\circ \al_n(x_0) n^*n(x_0) \\
	&= <(f\circ \al_n(x_0) )n, n> 
\end{align*}
and so $|<[fn], [n]>| = ||[fn]||_2\, |||n||_2.$ It follows from the converse direction of Cauchy-Schwarz that $[fn]$ is a scalar multiple of $[n],$ and by the above that
the scalar is $f\circ \al_n(x_0).$  The calculation that $[nf] = f(x_0) [n]$ is similar.
\end{proof}

Define an equivalence relation on $\sN(x_0) $ by: $n_1 \sim n_2$ if and only if there are functions $h_1, h_2 \in C(X)$ with $h_1(x_0),\  h_2(x_0) \neq 0,$ such that
$n_1 h_1 = n_2 h_2.$ 

\begin{lemma} \label{l:equivalent normalizers}
Let $n_1, n_2 \in \sN(x_0),$ and suppose that 
$[n_1]$ is a nonzero vector. Then $n_1 \sim n_2$ if and only if there is a constant $c \in \bbC, c \neq 0,$ such that  $[n_1] = c [n_2].$
\end{lemma}

\begin{proof}
Suppose $n_1 \sim n_2$ and let $h_1, h_2$ be as in the definition. Then for any vector $v \in H_0,$ 
\begin{align*}
< [n_1], v> = \frac{1}{h_1(x_0)}<[n_1 h_1], v> \\
	&= \frac{1}{h_1(x_0)}<[n_2 h_2], v> \\
	&= \frac{h_2(x_0)}{h_1(x_0)}<[n_2], v>
\end{align*}
Since $v \in H_0$ was arbitrary, it follows that $ [n_1] = c[n_2]$ with $c =  \frac{h_2(x_0)}{h_1(x_0)} \neq 0.$

For the converse, suppose there is $c \in \bbC, c \neq 0,$ such that $[n_1] = c[n_2].$ In particular, for all $m \in \sN,$
\[ <[n_1], [m]> = <c[n_2], [m]> \text{ or } E(m^*n_1)(x_0) = E(m^*cn_2)(x_0) .\]
Taking $m = n_1, cn_2,$ yields that $n_1^*n_1(x_0) = |c|^2 n_2^*n_2(x_0) $ and since the vector $[n_1] \neq 0, \ n_1^*n_1(x_0) > 0.$

Also $<[a], [n_1]> = <[a], [cn_2]>$ for any $a \in \fA,$ so that $E(n_1^*a)(x_0) = E((cn_2)^*a)(x_0),$ and utilizing the result above,
\[ \hat{a}(\al_{n_1}(x_0), n_1, x_0) = \hat{a}(\al_{n_2}(x_0), cn_2, x_0), \ a \in \fA \]
so that, in particular $\al_{n_1}(x_0) = \al_{n_2}(x_0),$ and denoting this point by $y,$
\[ [y, n_1, x_0] = [y, cn_2, x_0]\]

From \cite{JRen08}, Section 4, there are functions $b_1, b_2 \in C(X)$ with \mbox{$b_1(x_0), b_2(x_0) > 0$} such that $n_1 b_1 = c n_2 b_2.$
Thus, taking $h_1 = b_1,  \ h_2 = cb_2$ we have that $h_1(x_0), h_2(x_0) \neq 0$ and $ n_1 h_1 = n_2 h_2,$ so that $n_1 \sim n_2.$

\end{proof}

\begin{corollary} \label{c:not sim}
For nonzero $n_1, n_2 \in \sN(x_0),$ either $n_1 \sim n_2$ or $[n_1] \perp [n_2].$
\end{corollary}

\begin{proof}
If $n_1 \nsim n_2,$ then by Lemma~\ref{l:equivalent normalizers} there do not exist functions $h_1, h_2 \in C(X), \ h_1(x_0), h_2(x_0) \neq 0,$ such that
$n_1h_1 = n_2 h_2.$ A fortiori there do not exist functions $b_1, b_2 \in C(X), $ and $z \in \bbT,$ with $b_1(x_0), b_2(x_0) > 0$ and
$n_1 b_1 = z n_2 b_2.$ Thus, $[\al_{n_1}(x_0), n_1, x_0]$ and $[\al_{n_2}(x_0), zn_2, x_0]$ are distinct points of the twised groupoid $(G, \Si),$
for all $z \in \bbT.$ As $\Si/\bbT= G,$ it follows that $[\al_{n_1}(x_0), \al_{n_1}, x_0],$ and $ [\al_{n_2}(x_0), \al_{n_2}, x_0]$ are distinct elements of the groupoid $G,$
so that the germs of $\al_{n_2}, \al_{n_2}$ are distinct at $x_0.$ From Lemma~\ref{l:hat n}, $E(n_1^*n_2)(x_0) = 0.$ In other words,
$<[n_1], [n_2]> = 0.$
\end{proof}

In each equivalence class of $\sN(x_0),$ choose a representative which is a unit vector. Since the set of equivalence classes is countable, we may enumerate the collection
$\{n_0 = E(n_0), n_1, n_2, \dots \}.$

\begin{lemma} \label{l:on basis}
The set $\{[n_i]: i = 0, 1, 2 \dots\}$ is an orthonormal basis for $H_0.$
\end{lemma}

\begin{proof} By assumption the $\{[n_i]\}$ are unit vectors, and since $n_i, n_j$ belong to different equivalence classes for $i \neq j,$ it follows from
Corollary~\ref{c:not sim} that  they are orthonormal.  Now it is enough to show that the linear span is dense.  Since $\fA_c$ is dense in $\fA,$
it is enough to show that for $a \in \fA_c,$ that $[a] \in \text{span}\{[n_i]\}.$

Let $a \in \fA_c.$ Hence
$a$ has the form $a = \sum_{k=0}^N m_k,\ m_k \in \sN.$ Now either $[m_k] = 0,$ which is the case if $m_k^*m_k(x_0) = 0,$ or else $m_k \in \sN(x_0),$ in which case
 $[m_k] = \be_k [n_{i_k}]$ for some $\be_k \in \bbC, \ i_k \in \bbN.$ Thus, $[a] \in \text{span}\{[n_i]\}.$
\end{proof}

Next we want to show that $\pi_0(\fM)$ is weakly dense in the (discrete) masa $\ell^{\infty}(\{n_i\}).$ To this end, let $c \in \ell^{\infty}(\{n_i\}),\ c = (c_i)$ and 
$F \subset \{0, 1, 2, \dots\}$ a finite subset. We claim there is an $f \in C(X) $ with $\pi_0(f)[n_i] = c_i [n_i], \ i \in F$ and $||f|| \leq ||c||_{\infty}.$

Now $\pi_0(f)[n_i] = [fn_i] = f\circ \al_{n_i}(x_0) [n_i]. $ Thus we need to obtain $f \in C(X)$ with $f\circ \al_{n_i}(x_0) = c_i, \ i \in F,$ and $||f|| \leq ||c||_{\infty}.$
But  $\{ \al_{n_i}(x_0): i \in F\}$ is a finite set of points of $X$ which are distinct because $x_0$ has trivial isotropy.
Thus the claim follows from the Tietze Extension Theorem. Denote the function $f = f_F.$

Now the net $\{f_F\}_F,$ where the finite subsets $F$ are partially ordered by inclusion, converges in the weak operator topology to $c.$
This proves
\begin{corollary} \label{c:fM is a masa}
The weak operator closure of $\pi_0(\fM)$ is a discrete masa in $B(H_0).$
\end{corollary}

\begin{remark} \label{r:masa preserving}
We note that it is in general not the case that if $\fA$ is a unital C$^*$-algebra with distinguished masa $\fM,$ and $\pi: \fA \to \sB(H)$
an irreducible representation, that $\pi(\fM)$ is weakly dense in a masa of $\sB(H).$  An example of this phenomenon is given in \cite{OrrPet95}, section IV.3.
\end{remark}

Now combining the results of Corollary~\ref{c:fM is a masa} together with Proposition~\ref{p:dirichlet} we obtain:

\begin{theorem} \label{t:Dirichlet}
Let $\sA$ be a unital operator algebra and $\fA$ a C$^*$-cover such that $\sA$ is Dirichlet in $\fA.$
  Suppose that $\sA$ contains an abelian $*$-subalgebra $\fM,$ which, under the natural
embedding of $\sA \mapsto \fA,$ is a Cartan  subalgebra of $\fA.$ Let $\pi_0$ be the irreducible representation of $\fA$ from Corollary~\ref{c:fM is a masa}.
Then the restriction of $\pi_0$ to $\sA,\  \pi_0: \sA \to \sB(H_0)$ is a nest representation. 
\end{theorem}

\begin{remark} \label{r:other representations}
Note that Proposition~\ref{p:dirichlet} only requires that the representation $\pi$ satisfy the condition that $\ol{\pi(\sA)}^{wk}$ contain a masa, whereas
the representations $\pi_0$ of Theorem~\ref{t:Dirichlet} satisfy the stronger condition that there is a masa $\fM \subset \sA$ whose image under $\pi_0$
is weakly dense in a masa in $\sB(H_{\pi_0}).$ Thus there are possibly additional irreducible representations $\pi$ of $\fA$ with $\text{Lat}(\pi(\sA))$ a nest in $\sB(H_{\pi}).$
\end{remark}

\begin{lemma} \label{l:sufficiently many}
Let $\fA$ be a unital, seperable C$^*$-algebra, and $\fM$ a masa in $\fA$ such that the pair $(\fA, \fM)$ is a Cartan pair. 
If $X = \hat{\fM},$ let $Y$ denote the subset of $X$ consisting of points with trivial isotropy. For $y \in Y,$ let $\rho_y$ denote the pure state
$\rho_y(a) = E(a)(y),$ and $\pi_y$ the corresponding irreducible representation. Then the representation
\[ \Pi(a) := \bigoplus_{y \in Y} \pi_y(a) \]
is isometric.
\end{lemma}

\begin{proof}

 We show that $\Pi$ is faithful by showing that any positive element in $\text{ker}(\Pi)$ is $0.$

By way of contradiction, we assume that $a \in \text{ker}(\Pi)$ is positive and nonzero. 
Since $E$ is a faithful conditional expectation, $E(a)$ is positive and nonzero.  By Corollary~\ref{c:ker pi0 cap C(X)},
$E(a)$ is positive element in $\text{ker}(\Pi) \cap C(X).$ By the same Corollary and the definition of $\Pi, \ E(a)$ vanishes on the pseudo-orbit $\sG(y)$
for all $y \in Y.$ But by \cite{JRen08}, Proposition 3.1, the set of points $Y$ with trivial isotropy is dense in $X.$ Hence, $E(a)$ vanishes on a dense subset,
and thus is identically zero.  It follows that $a = 0,$ and therefore that $\Pi$ is faithful.

\end{proof}

One of the earliest results in the theory of C$^*$-algebras was that there are 'sufficiently many' irreducible representations. Here we prove an analogue
for Dirichlet algebras. 

\begin{corollary} \label{c:sufficiently many nest}
Let $\sA, \fM, \fA$ be as in Theorem~\ref{t:Dirichlet}, and $Y$ and $\pi_y$ as in Lemma~\ref{l:sufficiently many}.
Then the Dirichlet algebra $\sA$ has sufficiently many nest representations. That is, for all $y \in Y, $ the restriction to $\sA$ of the representation $\pi_y$
is a nest representation of $\sA,$ and
\[ ||a|| = \sup_{y \in Y} ||\pi_y(a)|| \text{ for all } y \in \sA.\]

\end{corollary}

\begin{proof}
This follows directly from Theorem~\ref{t:Dirichlet} and Lemma~\ref{l:sufficiently many}
\end{proof}


\section{Examples} \label{s:examples}

In this section we present several examples of non-self adjoint operator algebras that have been considered by various authors that we now know satisfy the conditions of
Corollary~\ref{c:sufficiently many nest}.

\begin{example} \label{e:TAF}
Let $\sA$ be a strongly maximal TAF algebra. (cf \cite{SPow92}) 
If $\fA$ is the inducitve limit of finite dimensional C$^*$-algebras $\fA_n,$ matrix units for $\fA_n$ are choosen so that each matrix unit of $\fA_n$ is a sum of matrix units
of $\fA_{n+1},$ and so that  $\sA_n := \sA \cap \fA_n $  embeds into $\sA_{n+1},$ and $\sA_n + \sA_n^* = \fA_n.$
 The diagonal $\sD_n := \sA_n \cap \sA_n^*$ of $\fA_n$ embeds into the diagonal $\sD_{n+1}$ of $\fA_{n+1}.$ 
The inductive limit $\sA$ of the $\sA_n$ is then a Dirichlet subalgebra of the AF algebra $\fA$ which contains the Cartan masa $\sD = \sA \cap \sA^*.$
Such masas $\sD$ were first constructed by Str\v{a}til\v{a} and Voiculescu in \cite{StrVoi75}.

As observed in \cite{JRen08}, all the ``privileged'' Cartan subalgebras of a given AF C$^*$algebra $\fA$ obtained by 
the Str\v{a}til\v{a}-Voiculescu construction are conjugate by an automorphism of the ambient
AF-algebra. However, it is not the case that the TAF algebras $\sA$ obtained by this construction are conjugate. Indeed, some may be simple Banach algebras, such as those
arising from the sequence of standard embeddings (e.g., \cite{Don93}), while others may have nontrivial radicals, such as those arising from nest embeddings.

\end{example}

\begin{example} \label{e:semicrossed product}
The semicrossed product $\sA$ associated with a compact metric space $X$ and a 
homeomorphism $\vpi$ is the closed subalgebra of the C$^*$-crossedproduct $\fA = C(X)\times_{\vpi}\bbZ$, which is generated by $C(X) $
 and nonnegative powers of the  canonical  unitary $U$.  Here $U$ satisfies $U^{-1}fU =f\circ \vpi, \ f \in C(X).$
The formal polynomials $\sum_{k \geq 0}  U^k f_k $ (with $f_k = 0$ except for finitely many $k$) are dense in the semicrossed product $\sA = C(X)\times_{\vpi}\bbZ^+.$

$C(X)$ is embedded in the crossed product as a Cartan subalgebra.
The Weyl groupoid $G$ associated with the pair $(\fA, C(X))$ consists of $\{ [x, \vpi^n, y]: x, y \in X, x = \vpi^n(y), \ n \in \bbZ \}.$  
The semicrossed product $\sA$ can be identified with functions supported on the semigroupoid $\{ [x, \vpi^n, y]: n \geq 0\}.$

If two dynamical systems $(X, \vpi),\ (X', \vpi')$ are conjugate, then it is not hard to show that there is an isomorphism of the respective crossed products,
which restricts to a (completely isometric) isomorphism of the two semicrossed products.  However, 
as noted in \cite{JRen08}, if two crossed products $\fA = C(X)\times_{\vpi} \bbZ$ and $\fA' = C(X')\times_{\vpi'}\bbZ$ are isomorphic, little can be deduced about
the dynamical systems $(X, \vpi)$ and $(X', \vpi'),$ even in the case of Cantor minimal systems. Indeed, the two systems need not even be flip conjugate.

On the other hand, if two  semicrossed  product  algebras $\sA,\ \sA'$ are isomorphic as \emph{algebras} then dynamical systems are conjugate. Of course this
implies that if $\sA, \ \sA'$ are isomorphic as operator algebras or as Banach algebras then the corresponding dynamical systems are conjugate.
This is the result of various authors, e.g., \cite{HadHoo88}, \cite{SPow92b}, \cite{DK.Con08}

\end{example}

\begin{example} \label{e:Cuntz}
A Cuntz  family is a set of $d \geq 2$ isometries $S_1, \dots, S_d$ in a Hilbert space $H$ satisfying
$\sum_{j=1}^d S_j S_j^* = I,$ and the C$^*$-algebra generated by the $S_j$ is Cuntz algebra $\sO_d.$
If $\mu, \  \nu$ are multi-indices of the form $i_1, \dots i_n, $  with $i_j \in \{1, \dots, d\}$ and $n \in \bbN,$
the subalgebra of finite sums of ``monomials'' of the form $S_{\mu}S_{\nu}^*$ is dense in $\sO_d.$
Elements of the form $S_{\mu} S_{\mu}^*$ generate a Cartan masa in $\sO_d.$

If $\mu = (i_1, \dots, i_n),$ let $|\mu| = n.$  Then we can define the subalgebra $\sA$ to be the norm closure in $\sO_d$
of linear combinations of elements of the form $S_{\mu}S_{\nu}^*$ with $|\mu| \geq |\nu|.$  Observe that if $|\mu| \geq |\nu|$ and $|\om| \geq |\eta|$ that
$(S_{\mu}S_{\nu}^*)(S_{\om}S_{\eta}^*)$ is either zero or has the form $S_{\al}S_{\be}^*$ with $|\al| - |\be| = |\mu| - |\nu| + |\om| - |\eta| \geq 0.$
Thus $\sA$ is a subalgebra, and the Dirichlet condition is clear. In this case $\sA \cap \sA^*$ is generated by $S_{\mu}S_{\nu}^*,\ |\mu| = |\nu|,$
(which is the UHF algebra $d^{\infty}$), so that $\sA \cap \sA^*$ properly contains the Cartan masa.

If $X = \{1, \dots, d\}^{\bbN}$ (with the product topology), then the groupoid $G = G(X, T)$ is associated with the one-sided shift $T: X \to X$:
\[ G = (x, m-n, y): x, y \in X, T^m(x) = T^n(y) \]
$G$ is essentially principal but not principal.
The Dirichlet algebra $\sA \subset \sO_d$ can be identified with functions supported on the ``sub-semigroupoid'' 
$P = \{(x, m-n, y) \in G: m - n \geq 0\}.$ (cf \cite{Dea95}, \cite{JRen08})

\end{example}

\begin{example} \label{e:graph}
There is a class of examples of Dirichlet subalgebras of graph C$^*$-algebras. We treat this very briefly, and refer the reader to \cite{IRae05} and \cite{HopPetPow05}
for more information. In this context, Cuntz isometries are replaced by Cuntz-Krieger partial isometries, satisfying the condition
\[ S_e^*S_e = \sum_{r(f) = s(e)} S_f S_f^* \]
where $e, f$ represent edges in the graph, such that the range of $e$ is the source of $f.$

The groupoid $G$ can be described as $\{ (x, m-n, y): T^m(x) = T^n(y)\},$ where $x, y$ infinte paths with range but no source, and $T$ is the one-sided shift.
In order for the functions supported on the unit space $G^{(0)}$ to be a masa imposes a condition on the graph, that that every loop have an entrance. (This condition
depends on the way composition of paths is defined; for some authors the condition becomes `every loop has an exit'.)

We define the Dirichlet subalgebra $\sA$ to be the subalgebra generated by the Cuntz-Krieger partial isometries $S_{\mu}S_{\nu}^*$ where $\mu, \nu$ are
multiindices with $|\mu| \geq |\nu|.$ To apply our results, we need to know that $X$ is compact, which implies a finiteness condition on the vertices.
\end{example}

\begin{example} \label{e:construction}
We sketch a procedure for generating a variety of examples. Suppose $(\fA, \fM)$ is a Cartan pair and $\sG$ is the Weyl pseudogroup.  Suppose $\sS \subset \sG$
is a subset with the following properties:
\begin{enumerate}
\item $\sS$ is closed under partial composition \\
\item $\sS \cap \sS^{-1} = \sG^{0}$ \\
\item $\sS \cup \sS^{-1} = \sG$
\end{enumerate}
where $\sG^{0}$ is the set of partial identities in $\sG,$ and $\sS^{-1}$ is the set of partial inverses in $\sS.$
Then we can define a Dirichlet algebra $\sA$ to be the algebra generated by those normalizers $n$ such that $\al_n \in \sS.$
The algebra $\sA^*$ is generated by normalizers $m$ so that $\al_{m^*} = \al_m^{-1} \in \sS,$ so that $\sA + \sA^*$ is dense in the C$^*$-algebra $\fA.$

The algebras $\sA$ constructed in this manner are Dirichlet in the strong sense, that $\sA \cap \sA^* = \fM$ and furthermore the conditional expectation
$E: \fA \to \fM$ is multiplicative on $\sA.$

Note that Examples~\ref{e:TAF} and \ref{e:semicrossed product} are examples of this type. On the other hand, observe that in  Examples~\ref{e:Cuntz} and \ref{e:graph}
the conditional expectation $E$ is not mutiplicative on the Dirichlet algebra.
\end{example}


\end{document}